\documentclass[a4paper]{amsart}
\usepackage[utf8]{inputenc}
\usepackage[T1]{fontenc}
\usepackage[charter]{mathdesign}

\usepackage{hyperref}
\hypersetup{
bookmarks=true,
pdfpagemode=UseNone,
pdfstartview=FitH,
pdfdisplaydoctitle=true,
pdfborder={0 0 0}, % No link borders
pdftitle={Norm convergence of multiple ergodic averages on amenable groups},
pdfauthor={Pavel Zorin-Kranich},
pdfsubject={Mathematics Subject Classification (2010): Primary 37A30},
pdfkeywords={multiple ergodic averages, amenable group, nilpotent group, polynomial mapping, commuting actions},
pdflang=en-US
}
\usepackage[style=alphabetic,url=false]{biblatex}

\numberwithin{equation}{section}
\newtheorem{theorem}[equation]{Theorem}
\newtheorem*{theorem*}{Theorem}
\newtheorem{multvNtheorem}[equation]{Multiplicator von Neumann Theorem}
\newtheorem{structuretheorem}[equation]{Structure Theorem}
\newtheorem{inversetheorem}[equation]{Inverse Theorem}
\newtheorem{lemma}[equation]{Lemma}
\newtheorem{proposition}[equation]{Proposition}

\theoremstyle{definition}
\newtheorem{definition}[equation]{Definition}
\theoremstyle{remark}

\newcommand*{\N}{\mathbb{N}}
\newcommand*{\Z}{\mathbb{Z}}
\newcommand*{\R}{\mathbb{R}}
\newcommand*{\dif}{\mathrm{d}}
\newcommand*{\inv}{^{-1}}
\newcommand*{\E}{\mathbb{E}}
\newcommand*{\from}{\colon}
\def\<{\left\langle}
\def\>{\right\rangle}

\newcommand*{\T}{\mathbb{T}}
\DeclareMathOperator{\complexity}{cplx}
\newcommand*{\cplx}{\mathbf{c}}
\newcommand*{\vd}{d}

\newcommand*{\m}{\mathrm{m}}
\newcommand{\meas}[1]{\m(#1)}

\newcommand{\aveN}{\frac{1}{N}\sum_{n=1}^N}
\newcommand{\Fo}{\Phi}
\newcommand{\ave}[2]{\E_{#1 \in #2}}

\newcommand{\AG}{\mathcal{G}} % Amenable group
\newcommand{\AGg}{g} % Element of \AG
\newcommand{\AGh}{h} % Element of \AG
\newcommand{\AGl}{\mathfrak{l}} % Element of \AG
\newcommand{\AGr}{\mathfrak{r}} % Element of \AG
\newcommand{\TG}{\mathcal{T}} % Transformation group
\newcommand*{\fTG}[1][]{\TG_{\bullet #1}} % Filtered transformation group
\newcommand{\TGT}{T} % Element of \TG
\newcommand{\TGS}{S} % Element of \TG
\newcommand*{\TGTvec}{\mathbf{\TGT}}
\newcommand*{\TGSvec}{\mathbf{\TGS}}

\newcommand*{\Trans}{R} % Left/right translation (maybe use L/R ?)
\newcommand{\Der}{D} % Discrete derivative
\newcommand*{\Av}[2][\TGTvec]{\mathcal{A}^{#1}_{#2}} % Average

\newcommand{\indx}{\alpha}
\newcommand{\indy}{\beta}
\newcommand{\iset}{A}

\newcommand{\folner}{F\o{}lner}

\usepackage{xparse}
\DeclareDocumentCommand\reduction{ m m g g }{%
  {\<#1 | #2\>_{\IfNoValueF{#3}{#3}%
      \IfNoValueF{#4}{,#4}}%
    }%
  }

\begin{document}

\title{Norm convergence of multiple ergodic averages on amenable groups}
\author{Pavel Zorin-Kranich}
\address
{Institute of Mathematics\\
Hebrew University, Givat Ram\\
Jerusalem, 91904, Israel}
\email{pzorin@math.huji.ac.il}
\urladdr{http://math.huji.ac.il/~pzorin/}
\begin{abstract}
We apply Walsh's method for proving norm convergence of multiple ergodic averages to arbitrary amenable groups.
We obtain convergence in the uniform Ces\`aro sense for their polynomial actions and for ``triangular'' averages associated to commuting homomorphic actions.
The latter generalizes a result due to Bergelson, McCutcheon, and Zhang in the case of two actions.
\end{abstract}
\maketitle

\section{Introduction}
In a recent breakthrough Walsh proved the norm convergence of nilpotent polynomial ergodic averages arising in Leibman's Szemer\'edi theorem for nilpotent groups \cite{MR1650102}.
His result may be stated as follows.
\begin{theorem*}[{\cite{MR2912715}}]
Let $(X,\mu)$ be a probability space and $\TG$ be a nilpotent group of $\mu$-preserving algebra automorphisms of $L^{\infty}(X)$.\footnote{Of course, such automorphisms arise from measure-preserving transformations on $(X,\mu)$.}
Then for any $f_{1},\dots,f_{k} \in L^{\infty}(X)$ and any polynomial\footnote{Polynomials into nilpotent groups will be defined in \textsection\ref{sec:polynomial}.} maps $\TGT_{1},\dots,\TGT_{k} : \Z \to \TG$ the averages
\[
\aveN \TGT_{1}(n)f_{1} \cdots \TGT_{k}(n)f_{k}
\]
converge in $L^{2}(X)$.
\end{theorem*}
This result has a long history, the main milestones being due, in roughly chronological order, to Furstenberg \cite{MR0498471}, Conze and Lesigne \cite{MR788966}, Host and Kra \cite{MR2150389}, Ziegler \cite{MR2257397}, Leibman \cite{MR2151605}, also jointly with Bergelson \cite{MR1881925}, Tao \cite{MR2408398}, Austin \cite{MR2599882}, and Host \cite{MR2539560}.

The purpose of this article is to show how the method underlying Walsh's proof in fact yields stronger results.
Firstly, we obtain convergence in the uniform Ces\`aro sense.
This provides a conceptually satisfying explanation for the syndeticity of the set of recurrence times in Leibman's nilpotent Szemer\'edi theorem.
Secondly, we replace $\Z$ by an arbitrary locally compact amenable group $\AG$.
This is motivated by a conjecture of Bergelson, McCutcheon, and Zhang \cite{MR1481813} regarding multiple recurrence for several commuting measure-preserving $\AG$-actions.
We note that this conjecture was resolved in the positive by Austin \cite{arXiv:1309.4315} after the completion of this work.

In order to formulate the main result and to facilitate the reading of the paper we will now introduce the standing notation that will remain unchanged throughout the text.
We denote by $\AG$ a locally compact (not necessarily second countable) amenable group with a left Haar measure $\m$.
We fix a probability space $(X,\mu)$ and a group $\TG$ of unitary operators on $L^{2}(X)$ that act as isometric algebra homomorphisms on $L^{\infty}(X)$.
Of course, one could work with the corresponding measure-preserving transformations on $X$ instead, but this would introduce notation overhead since Walsh's method is operator-theoretic in nature.
All maps $\TGT_{i}:\AG\to\TG$ that we consider are assumed to be measurable.
\begin{theorem}
\label{thm:norm-convergence}
Let $f_{0},\dots,f_{k}\in L^{\infty}(X)$ and $\TGT_{1},\dots,\TGT_{k} : \AG \to \TG$.
Suppose that one of the following holds.
\begin{enumerate}
\item\label{thm:norm-convergence:polynomial}
The group $\TG$ is nilpotent and the maps $\TGT_{i}$ are polynomial or
\item\label{thm:norm-convergence:commuting}
the maps $\TGT_{i}$ have the form $\TGT_{i}=\prod_{j=1}^{i} \TGS_{j}$, where $\TGS_{j}$ are commuting antihomomorphisms.
\end{enumerate}
Then for every left \folner{} net $(\Fo_{\indx})_{\indx\in\iset}$ the limit
\[
\lim_{\indx} \meas{\Fo_{\indx}}\inv \int_{\Fo_{\indx}} f_{0} \TGT_{1}(\AGg) f_{1} \cdots \TGT_{k}(\AGg) f_{k} \dif\m(\AGg)
\]
exists in $L^{2}(X)$ and does not depend on the \folner{} net.
\end{theorem}
In order to illustrate the power, but also the limitations, of Theorem~\ref{thm:norm-convergence}\eqref{thm:norm-convergence:polynomial} we note that it provides convergence of the averages in \cite[Theorem 1.2]{arXiv:1105.5612} on the joining (and not only of their expectations on the first factor) but fails to produce the invariance.
An analog of Theorem~\ref{thm:norm-convergence}\eqref{thm:norm-convergence:polynomial} cannot hold for solvable groups of exponential growth in view of counterexamples due to Bergelson and Leibman \cite{MR2041260}.
Theorem~\ref{thm:norm-convergence}\eqref{thm:norm-convergence:commuting} generalizes \cite[Theorem 4.8]{MR1481813}, which is the $k=2$ case of it.
Different proofs of Theorem~\ref{thm:norm-convergence}\eqref{thm:norm-convergence:polynomial} and \eqref{thm:norm-convergence:commuting} for discrete groups $\AG$ were recently obtained by Austin in \cite{arxiv:1310.3219} and \cite{arXiv:1309.4315}, respectively.

Walsh's argument uses Kreisel's \emph{no-counterexample interpretation} \cite{MR0049135} of convergence.
In order to illustrate this and some other ideas involved in Walsh's technique, we begin with a proof of a quantitative version of the von Neumann mean ergodic theorem for multiplicators on the unit circle $\T$.

\section{A close look at the von Neumann mean ergodic theorem}
\label{sec:vn}
Throughout this section $\mu$ denotes a Borel measure on $\T$ and $Uf(\lambda) = \lambda f(\lambda)$ is a multiplicator on $L^2(\T,\mu)$.
The von Neumann mean ergodic theorem in its simplest form reads as follows.
\begin{multvNtheorem}
\label{thm:vn-mult}
Let $\mu$ and $U$ be as above.
Then the ergodic averages $a_{N} = \E_{n\leq N} U^n f$ converge in $L^2(\mu)$.
\end{multvNtheorem}
\begin{proof}
The averages $a_{N}$ are dominated by $|f|$ and converge pointwise, namely to $f(1)$ at $1$ and to zero elsewhere.
\end{proof}

It is well-known that no uniform bound on the rate of convergence of the ergodic averages can be given even if $U$ is similar to the Koopman operator of a measure-preserving transformation \cite{MR510630}.
However, there does exist a uniform bound on the \emph{rate of metastability} of the ergodic averages.
Let us recall the concept of \emph{metastability}.
\index{metastability}
The sequence $(a_N)$ converges if and only if it is Cauchy, i.e.
\[
\forall\epsilon>0 \, \exists M \, \forall N,N' (M \leq N,N' \implies \|a_N - a_{N'}\|_2 < \epsilon).
\]
The negation of this statement, i.e. ``$(a_N)$ is not Cauchy'' reads
\[
\exists\epsilon>0 \, \forall M \, \exists N,N' \colon M \leq N, N' ,\, \|a_N - a_{N'}\|_2 \geq \epsilon.
\]
Choosing witnesses $N(M)$, $N'(M)$ for each $M$ and defining
\[
F(M) = \max\{N(M), N'(M)\}
\]
we see that this is equivalent to
\[
\exists\epsilon>0 \, \exists F \from\N\to\N \, \forall M \, \exists N,N' \colon M \leq N,N' \leq F(M) ,\, \|a_N - a_{N'}\|_2 \geq \epsilon.
\]
Negating this we obtain that $(a_N)$ is Cauchy if and only if
\[
\forall\epsilon>0 \, \forall F \from\N\to\N \, \exists M \, \forall N,N'
(M \leq N,N' \leq F(M) \implies \|a_N - a_{N'}\|_2 < \epsilon ).
\]
This kind of condition, namely that the oscillation of a function is small on a finite interval is called \emph{metastability}.
A bound on the \emph{rate of metastability} is a bound on $M$ that may depend on $\epsilon$ and $F$ but not the sequence $(a_{N})_{N}$.

The appropriate reformulation of the von Neumann mean ergodic theorem for the operator $U$ in terms of metastability reads as follows.
\begin{multvNtheorem}[finitary version]
\label{thm:vn-fin}
Let $\mu$ and $U$ be as above. Then for every $\epsilon>0$, every function $F\from\N\to\N$ and every $f \in L^{2}(\mu)$ there exists a number $M$ such that for every $M \leq N, N' \leq F(M)$ we have
\begin{equation}
\label{eq:vn}
\Big\| \E_{n \leq N} U^n f - \E_{n \leq N'} U^n f \Big\|_2 < \epsilon.
\end{equation}
\end{multvNtheorem}
Although Theorem~\ref{thm:vn-fin} is equivalent to Theorem~\ref{thm:vn-mult} by the above considerations, we now attempt to prove it as stated.

\begin{proof}[Proof of Theorem~\ref{thm:vn-fin}]
It clearly suffices to consider strictly monotonically increasing functions $F$.
Let us assume $\|f\|_{2}=1$, take an arbitrary $M$ and see what can be said about the averages in \eqref{eq:vn}.

Suppose first that $f$ is supported near $1$, say on the disc $A_{M}$ with radius $\frac{\epsilon}{6 F(M)}$ and center $1$.
Then $U^n f$ is independent of $n$ up to a relative error of $\frac{\epsilon}{6}$ provided that $n \leq F(M)$, hence both averages are nearly equal.

Suppose now that the support of $f$ is bounded away from $1$, say $f$ is supported on the complement $B_{M}$ of the disc with radius $\frac{12}{\epsilon M}$ and center $1$.
Then the exponential sums $\frac1N \sum_{1 \leq n \leq N} \lambda^n$ are bounded by $\frac{\epsilon}{6}$ for all $\lambda$ in the support of $f$ provided that $N \geq M$, hence both averages are small.

However, there is an annulus whose intersection with the unit circle $E_M = \T \setminus ( A_M \cup B_M )$ does not fall in any of the two cases.
The key insight is that the regions $E_{M_i}$ can be made pairwise disjoint if one chooses a sufficiently rapidly growing sequence $(M_i)_i$, for instance it suffices to ensure $\frac{12}{\epsilon M_{i+1}} < \frac{\epsilon}{6 F(M_{i})}$.

Given $f$ with $\|f\|_2 \leq 1$, we can by the pigeonhole principle find an $i$ such that $\|f E_{M_i}\|_{2} < \epsilon/6$ (here we identify sets with their characteristic functions).
Thus we can split
\begin{equation}
\label{eq:decomposition-walsh}
f = \sigma + u + v,
\end{equation}
where $\sigma = f A_{M_i}$ is ``structured'', $u = f B_{M_i}$ is ``pseudorandom'', and $v = f E_{M_i}$ is $L^2$-small.
By the above considerations we obtain \eqref{eq:vn} for all $M_i \leq N, N' \leq F(M_i)$.
\end{proof}

Observe that the sequence $(M_i)_i$ in the foregoing proof does not depend on the measure $\mu$.
Moreover, a finite number of disjoint regions $E_{M_i}$ suffices to ensure that $f E_{M_i}$ is small for some $i$.
This yields the following strengthening of the von Neumann theorem.
\begin{multvNtheorem}[quantitative version]
\label{thm:vn-quan}
For every $\epsilon>0$ and every function $F\from\N\to\N$ there exist natural numbers $M_1, \dots, M_K$ such that for every $\mu$ and every $f \in L^{2}(\mu)$ with $\|f\|_{2}\leq 1$ there exists an $i$ such that for every $M_i \leq N, N' \leq F(M_i)$ we have
\[
\Big\| \E_{n \leq N} U^n f - \E_{n \leq N'} U^n f \Big\|_2 < \epsilon,
\]
where $Uf(\lambda)=\lambda f(\lambda)$ is a multiplicator as above.
\end{multvNtheorem}

The spectral theorem or the Herglotz-Bochner theorem can be used to deduce a similar result for any unitary operator.
The argument of Avigad, Gerhardy, and Towsner \cite[Theorem 2.16]{MR2550151} gives a similar result for arbitrary contractions on Hilbert spaces.
An even more precise result regarding contractions on uniformly convex spaces has been recently obtained by Avigad and Rute \cite{arxiv:1203.4124v1}.

Quantitative statements in the spirit of Theorem~\ref{thm:vn-quan} with uniform bounds that do not depend on the particular measure-preserving system allowed Walsh to use a certain induction argument that breaks down if this uniformity is disregarded.
A decomposition of the form \eqref{eq:decomposition-walsh}, albeit a much more elaborate one (Structure Theorem~\ref{thm:structure}), will also play a prominent role.

\section{\folner{} nets}
\label{sec:folner}
In this section we introduce the notation for various phenomena surrounding the \folner{} condition for amenability.
Recall that $\AG$ is a locally compact amenable group and $\m$ is a left Haar measure on $\AG$.
\begin{definition}
\label{def:folner-net}
A net $(\Fo_{\indx})_{\indx\in\iset}$ of non-null compact subsets of $\AG$ is called a \emph{\folner{} net} if
for every compact set $K \subset \AG$ one has
\[
\lim_{\indx} \sup_{\AGg\in K} \meas{\AGg\Fo_{\indx} \Delta \Fo_{\indx}}/\meas{\Fo_{\indx}} = 0.
\]
\end{definition}
It is more appropriate to call such objects \emph{left \folner{} nets}, but since we will not have to deal with the corresponding right-sided notion we omit the qualifier ``left''.
It is well-known that every amenable group admits a \folner{} net, which can be chosen to be a sequence if the group is $\sigma$-compact \cite[Theorem 4.16]{MR961261}.

By the \folner{} property for every $\gamma>0$ there exists a function $\varphi_{\gamma}\from \iset\to\iset$ such that
\begin{equation}
\label{eq:varphi}
\sup_{\AGg\in \Fo_{\indx}}\meas{\AGg\Fo_\indy \Delta \Fo_\indy} / \meas{\Fo_\indy} < \gamma
\text{ for every } \indy \geq \varphi_{\gamma}(\indx).
\end{equation}
Given a \folner{} net $(\Fo_{\indx})_{\indx\in\iset}$, we call sets of the form $\Fo_\indx \AGr$, $\AGr\in\AG$, $\indx\in\iset$, \emph{\folner{} sets}.
Such sets are usually denoted by the letter $I$.
For a \folner{} set $I$ we write $\lfloor I \rfloor = \indx$ if $I = \Fo_\indx \AGr$ for some $\AGr\in\AG$.

Note that, for any \folner{} net $(\Fo_{\indx})_{\indx}$ and any $\AGr_{\indx}\in\AG$, the net $(\Fo_{\indx}\AGr_{\indx})_{\indx\in\iset}$ is again \folner{}.
Thus we could replace the sets $\Fo_{\indx}$ in Theorem~\ref{thm:norm-convergence} by sets $\Fo_{\indx}\AGr_{\indx}$, but no additional generality would be gained by doing so.
It will nevertheless be crucial to work with estimates that are uniform over such right translates in the proof.
This is because the family of \folner{} sets is directed by \emph{approximate} (up to an arbitrarily small proportion) inclusion.
To be more precise, we say that a finite measure set $K$ is \emph{$\gamma$-approximately included} in a measurable set $I$, in symbols $K \lesssim_{\gamma} I$, if $\meas{K \setminus I}/\meas{K} < \gamma$.
\begin{lemma}
\label{lem:ceil}
For every $\gamma>0$ and any compact sets $I$ and $I'$ with positive measure
there exists an index $\lceil I, I' \rceil_{\gamma} \in\iset$ with the property that
for every $\indx \geq \lceil I, I' \rceil_{\gamma}$ there exists some $\AGr\in\AG$ such that
$I \lesssim_{\gamma} \Fo_\indx \AGr$ and $I' \lesssim_{\gamma} \Fo_\indx \AGr$.
\end{lemma}
We use the expectation notation $\E_{\AGg\in I} f(\AGg) = \meas{I}\inv \int_{\AGg\in I} f(\AGg) \dif\m(\AGg)$ for finite measure subsets $I \subset\AG$, where the integral is taken with respect to the left Haar measure $\m$.
Note that the expectation satisfies $\E_{\AGg\in \AGl I\AGr} f(\AGg) = \E_{\AGg\in I} f(\AGl\AGg \AGr)$ for any $\AGl,\AGr\in\AG$.
\begin{proof}
Let $K \subset \AG$ be compact and $c>0$ to be chosen later.
By the \folner{} property there exists an index $\indx_{0}\in \iset$ such that for every $\indx \geq \indx_{0}$ we have $\meas{\AGl\Fo_\indx \cap \Fo_\indx}/\meas{\Fo_\indx} > 1-c$ for all $\AGl\in K$.
Integrating over $K$ and using Fubini's theorem we obtain
\begin{align*}
1-c
&< \E_{\AGl\in K} \E_{\AGr\in \Fo_\indx} 1_{\Fo_\indx}(\AGl \AGr)\\
&= \E_{\AGr\in \Fo_\indx} \E_{\AGl\in K} 1_{\Fo_\indx \AGr\inv}(\AGl)\\
&= \E_{\AGr\in \Fo_\indx} \meas{K \cap \Fo_\indx \AGr\inv}/\meas{K}.
\end{align*}
Therefore there exists a $\AGr\in\AG$ (that may depend on $\indx \geq \indx_0$) such that $\meas{K \cap \Fo_\indx \AGr}/\meas{K} > 1-c$, so $\meas{K \setminus \Fo_\indx \AGr}/\meas{K} < c$.

We apply this with $K = I \cup I'$ and $c = \gamma \frac{\min\{\meas{I},\meas{I'}\}}{\meas{K}}$.
Let $\lceil I, I' \rceil_{\gamma} := \indx_{0}$ as above and $\indx \geq \lceil I, I' \rceil_{\gamma}$.
Then for an appropriate $\AGr\in\AG$ we have
\[
\meas{I \setminus \Fo_\indx \AGr}/\meas{I} \leq \meas{K \setminus \Fo_\indx \AGr}/\meas{I} < \meas{K}c/\meas{I} \leq \gamma,
\]
and analogously for $I'$.
\end{proof}
Any two \folner{} sequences are subsequences of some other \folner{} sequence, so if the Ces\`aro averages converge along every \folner{} sequence, then the limit does not depend on the \folner{} sequence.
For similar reasons this is also true for \folner{} nets; we include a proof for completeness.
\begin{lemma}
\label{lem:Clim-indep}
Suppose that $u:\AG\to V$ is a map into a Banach space such that for every \folner{} net $(\Fo_{\indx})_{\indx\in\iset}$ the limit $\lim_{\indx}\ave{\AGg}{\Fo_{\indx}}u(\AGg)$ exists.
Then the limit does not depend on the \folner{} net $\Phi$.
\end{lemma}
\begin{proof}
Let $(\Fo^{0}_{\indx})_{\indx\in\iset}$ and $(\Fo^{1}_{\indx'})_{\indx'\in\iset'}$ be \folner{} nets.
Replacing the index sets by $\iset\times\iset'$ with the product order if necessary we may assume $\iset=\iset'$.
If $\iset$ has a maximal element $\indy$, then by asymptotic invariance it follows that $\Fo_{\indy}=\Fo'_{\indy}=\AG$, and we are done.
Otherwise by \cite{MR0237386} there exists a partition $\iset = \iset^{0} \cup \iset^{1}$ into cofinal subsets, that is, subsets that contain a successor for any given element of $\iset$.
The net given by $\Fo_{\indx}:=\Fo^{i}_{\indx}$ if $\indx\in\iset^{i}$ is a \folner{} net and from cofinality it follows that
\[
\lim_{\indx}\ave{\AGg}{\Fo^{0}_{\indx}}u(\AGg)
= \lim_{\indx}\ave{\AGg}{\Fo_{\indx}}u(\AGg)
= \lim_{\indx}\ave{\AGg}{\Fo^{1}_{\indx}}u(\AGg).
\qedhere
\]
\end{proof}

\section{Complexity}
\label{sec:complexity}
In this section we give a streamlined treatment of Walsh's notion of \emph{complexity} \cite[\textsection 4]{MR2912715}.
It serves as the induction parameter in the proof of Theorem~\ref{thm:norm-convergence}.

We call an ordered tuple $\TGTvec = (\TGT_{0},\dots, \TGT_{j})$ of measurable mappings from $\AG$ to $\TG$ in which $\TGT_{0} \equiv 1_{\TG}$ a \emph{system} (it is not strictly necessary to include the constant mapping $\TGT_{0}$ in the definition, but it comes in handy in inductive arguments).

The complexity of the trivial system $\TGTvec = (1_{\TG})$ is by definition at most zero, in symbols $\complexity\TGTvec \leq 0$.
A system has finite complexity if it can be reduced to the trivial system in finitely many steps by means of two operations, \emph{reduction} (used in Proposition~\ref{prop:metastability}) and \emph{cheating} (used in Theorem~\ref{thm:metastability}).

Recall that the discrete derivative of a map $\TGT:\AG\to \TG$ is defined by
\[
\Der_{\AGr}\TGT(\AGg) = \TGT\inv(\AGg) \Trans_{\AGr}\TGT(\AGg),
\text{ where }
\Trans_{\AGr}\TGT(\AGg) = \TGT(\AGg\AGr).
\]
For $\AGr \in \AG$ the \emph{$\AGr$-reduction} of mappings $\TGT,\TGS \from\AG\to \TG$ is the mapping
\index{reduction}
\[
\reduction{\TGT}{\TGS}{\AGr}(\AGg)
= \Der_{\AGr}(\TGT\inv)(\AGg) \Trans_{\AGr}\TGS(\AGg)
= \TGT(\AGg)\TGT(\AGg\AGr)\inv \TGS(\AGg\AGr)
\]
and the $\AGr$-reduction of a system $\TGTvec = (\TGT_{0},\dots, \TGT_{j})$ is the system
\[
\TGTvec_{\AGr}^{*}
:=
\TGTvec' \uplus \reduction{\TGT_{j}}{\TGTvec'}{\AGr}
=
\left( \TGT_{0},\dots, \TGT_{j-1}, \reduction{\TGT_{j}}{\TGT_{0}}{\AGr},\dots,\reduction{\TGT_{j}}{\TGT_{j-1}}{\AGr} \right),
\]
where we use the shorthand notation $\TGTvec' = (\TGT_{0},\dots,\TGT_{j-1})$ and $\reduction{\TGT_{j}}{(\TGT_{0},\dots,\TGT_{j-1})} = (\reduction{\TGT_{j}}{\TGT_{0}},\dots,\reduction{\TGT_{j}}{\TGT_{j-1}})$, and where the symbol ``$\uplus$'' denotes concatenation.
If the reduction $\TGTvec_{\AGr}^{*}$ has complexity at most $\cplx-1$ for every $\AGr \in \AG$, then the system $\TGTvec$ is defined to have complexity at most $\cplx$.

Furthermore, if $\TGTvec$ is a system of complexity at most $\cplx$ and the system $\tilde\TGTvec$ consists of functions of the form $\TGT c$, where $\TGT\in\TGTvec$ and $c\in \TG$, then we cheat and set $\complexity\tilde\TGTvec \leq \cplx$.
This definition tells that striking out constants and multiple occurrences of the same mapping in a system as well as rearranging mappings will not change the complexity, and adding new mappings can only increase the complexity, for example
\[
\complexity(1_{\TG},\TGT_{2},\TGT_{1}c,\TGT_{1},c')
= \complexity(1_{\TG},\TGT_{1},\TGT_{2})
\leq  \complexity(1_{\TG},\TGT_{1},\TGT_{2},\TGT_{3}).
\]
Note that cheating is transitive in the sense that if one can go from system $\TGTvec$ to system $\tilde\TGTvec$ in finitely many cheating steps, then one can also go from $\TGTvec$ to $\tilde\TGTvec$ in one cheating step.

A generic system certainly does not have finite complexity.
We will describe two classes of systems that do have finite complexity, leading to the two cases in Theorem~\ref{thm:norm-convergence}.

\subsection{Commuting actions}
We begin with the simpler class of systems arising from commuting actions.
Recall that a map $\TGS\from\AG\to\TG$ is called an \emph{antihomomorphism} if $\TGS(\AGg\AGh)=\TGS(\AGh)\TGS(\AGg)$ for every $\AGg,\AGh\in\AG$.
Antihomomorphisms $\AG\to\TG$ correspond to measure-preserving actions of $\AG$ on $(X,\mu)$.
Two antihomomorphisms $\TGS_{i},\TGS_{j}\from\AG\to\TG$ are said to \emph{commute} if
\[
\TGS_{i}(\AGh)\TGS_{j}(\AGg) = \TGS_{j}(\AGg)\TGS_{i}(\AGh)
\quad\text{for any}\quad
\AGh,\AGg\in\AG.
\]
\begin{proposition}
\label{prop:comm-antihom}
Let $\TGS_{0}\equiv 1_{\TG}$ and $\TGS_{1},\dots,\TGS_{k} \from\AG\to \TG$ be antihomomorphisms that commute pairwise.
Then the system $(\TGS_{0},\TGS_{0}\TGS_{1},\dots,\TGS_{0}\dots \TGS_{k})$ has complexity at most $k$.
\end{proposition}
\begin{proof}
Every antihomomorphism $\TGS_{i}\from\AG\to \TG$ satisfies
\[
\Der_{\AGr}(\TGS_{i}\inv)(\AGg) = \TGS_{i}(\AGg) \TGS_{i}(\AGg\AGr)\inv = \TGS_{i}(\AGg) (\TGS_{i}(\AGr)\TGS_{i}(\AGg))\inv = \TGS_{i}(\AGr)\inv
\]
and
\[
\Trans_{\AGr}\TGS_{i}(\AGg) = \TGS_{i}(\AGg\AGr) = \TGS_{i}(\AGr)\TGS_{i}(\AGg).
\]
Thus for every $i<k$ we have
\begin{align*}
\reduction{\TGS_{0}\dots \TGS_{k}}{\TGS_{0}\dots \TGS_{i}}{\AGr}
&=
\Der_{\AGr} ((\TGS_{0}\dots \TGS_{k})\inv) \Trans_{\AGr}(\TGS_{0}\dots \TGS_{i})\\
&=
(\TGS_{0}(\AGr)\dots \TGS_{k}(\AGr))\inv \TGS_{0}(\AGr)\TGS_{0}\dots \TGS_{i}(\AGr)\TGS_{i}\\
&=
\TGS_{0}\dots \TGS_{i} \TGS_{i+1}(\AGr)\inv \dots \TGS_{k}(\AGr)\inv.
\end{align*}
Since $\TGS_{i+1}(\AGr)\inv \dots \TGS_{k}(\AGr)\inv \in \TG$ is a constant, we obtain
\[
\complexity (\TGS_{0},\TGS_{0}\TGS_{1},\dots,\TGS_{0}\dots \TGS_{k})^{*}_{\AGr}
=
\complexity (\TGS_{0},\TGS_{0}\TGS_{1},\dots,\TGS_{1}\dots \TGS_{k-1})
\]
by cheating.
We can conclude by induction on $k$.
\end{proof}

\subsection{Polynomial mappings}
\label{sec:polynomial}
Polynomials on $\Z$ of degree $\leq d$ can be characterized as those maps all of whose $(d+1)$-th discrete derivatives vanish identically.
A similar definition can be made for maps on any groups, but it has a serious disadvantage: unlike in the commutative case, a product of two polynomials of degree $\leq d$ may have higher degree $>d$, which causes difficulties in various inductive arguments.
This flaw has been rectified for polynomial mappings into nilpotent groups by Leibman who introduced the notion of vector degree and showed that polynomial mappings of a given vector degree form a group under pointwise operations \cite[Proposition 3.7 and erratum]{MR1910931}.
We find it more convenient to phrase his result in terms of filtrations.
A \emph{prefiltration} $\fTG$ is a sequence of subgroups
\begin{equation}
\label{eq:filtration}
\TG_{0} \geq \TG_{1} \geq \TG_{2} \geq \dots
\quad\text{such that}\quad
[\TG_{i},\TG_{j}]\subset \TG_{i+j}
\quad\text{for all } i,j\in\N=\{0,1,\dots\}.
\end{equation}
A \emph{filtration} (on a group $\TG$) is a prefiltration in which $\TG_{0}=\TG_{1}$ (and $\TG_{0}=\TG$).
A prefiltration is said to have \emph{length} $d\in\N$ if $\TG_{d+1}$ is the trivial group and length $-\infty$ if $\TG_{0}$ is the trivial group.
In this article we only consider prefiltrations for which one of these alternatives holds.
If $\TG$ is a nilpotent group, then the lower central series is a filtration.
For a prefiltration $\fTG$ of length $d$ we define $\fTG[+t]$ to be the prefiltration of length $d-t$ ($=-\infty$ if $d<t$) on $\TG$ given by $(\fTG[+t])_{i}=\TG_{i+t}$.

We define $\fTG$-polynomial maps by induction on the length of the prefiltration.
\begin{definition}
\label{def:polynomial}
Let $\fTG$ be a prefiltration of length $d\in\{-\infty\}\cup\N$.
A map $\TGT\from\AG\to \TG_{0}$ is called \emph{$\fTG$-polynomial} if either $d=-\infty$ (so that $\TGT\equiv 1_{\TG}$) or for every $\AGr\in\AG$ the discrete derivative $\Der_{\AGr}\TGT$ is $\fTG[+1]$-polynomial.
\end{definition}
Heuristically, this means that a map $\TGT:\AG\to \TG$ is polynomial if and only if every discrete derivative is polynomial ``of lower degree'' (although it usually does not make sense to define a ``degree'' for polynomials into nilpotent groups since it is necessary to keep track of the prefiltration $\fTG$ anyway).
Note that any map $\TGT\from\AG\to \TG$ into a nilpotent group that is polynomial of scalar degree $\leq d$ in the sense that
\[
\Der_{\AGr_{1}}\dots \Der_{\AGr_{d+1}} \TGT \equiv 1_{\TG}
\text{ for any }
\AGr_{1},\dots,\AGr_{d+1}\in\AG
\]
is also polynomial in the sense of Definition~\ref{def:polynomial}.
Indeed, if $\fTG$ is the lower central series, then
\begin{equation}
\label{eq:scalar-poly-filtration}
\TG_{0} \geq
\underbrace{\TG_{1} \geq \dots \geq \TG_{1}}_{d \text{ times}} \geq \dots \geq
\underbrace{\TG_{s} \geq \dots \geq \TG_{s}}_{d \text{ times}} \geq \TG_{s+1}
\end{equation}
is again a filtration, and $\TGT$ is polynomial with respect to it.
This is because the $(d+1)$-th derivative of $\TGT$ vanishes identically and in particular takes values in $\TG_{2}$, although one cannot a priori say anything about the $d$-th derivative.
On the other hand any mapping that is polynomial with respect to some filtration has scalar degree bounded by the length of the filtration.
This shows in particular that for any finite set of polynomial maps into a nilpotent group there is a filtration with respect to which all of them are polynomial.

The next theorem is a version of Leibman's result that polynomial mappings form a group \cite[Theorem 3.4]{MR1910931}.
A short proof can be found in \cite{arxiv:1206.0287}.
\begin{theorem}
\label{thm:poly-group}
The set of $\fTG$-polynomial maps $\AG\to \TG$ is a group under pointwise operations and is shift-invariant in the sense that for every $\fTG$-polynomial $\TGT$ and $\AGr\in\AG$ the translate $\Trans_{\AGr}g$ is also $\fTG$-polynomial.
Moreover, for any $\fTG[+t_{i}]$-polynomial maps $\TGT_{i} \from \AG \to \TG$, $i=0,1$, the commutator $[\TGT_{0},\TGT_{1}]$ is $\fTG[+t_{0}+t_{1}]$-polynomial.
\end{theorem} 
If $\AG=\Z^{r}$ or $\AG=\R^{r}$, then examples of polynomial mappings are readily obtained considering $\TGT(n)=T_{1}^{p_{1}(n)}\cdot\dots\cdot T_{l}^{p_{l}(n)}$, where $p_{i}\from\Z^{r}\to\Z$ (resp. $\R^{r}\to\R$) are conventional polynomials and $T_{i} : \Z \to \TG$ (resp. $\R\to \TG$) are one-parameter subgroups.
It is also known that group homomorphisms between any not necessarily commutative group and a nilpotent group are polynomial, see \cite{arxiv:1206.0287}.

We say that a system $(\TGT_{0},\dots,\TGT_{j})$ is $\fTG$-polynomial for a prefiltration $\fTG$ if every map $\TGT_{i}$ is $\fTG$-polynomial.
We now record a streamlined proof of Walsh's result that that every polynomial system has finite complexity.
For brevity we will denote discrete derivatives by
\[
\Der_{\AGr}\TGT(\AGg) := \TGT\inv(\AGg) \Trans_{\AGr}\TGT(\AGg),
\text{ where }
\Trans_{\AGr}\TGT(\AGg) = \TGT(\AGg\AGr).
\]
Note that for every $\fTG$-polynomial $\TGT$ and $\AGr\in\AG$ the translate $\Trans_{\AGr}\TGT$ is also $\fTG$-polynomial (since $\Trans \TGT=\TGT \Der_{\Trans}\TGT$).
We will omit the index $\AGr$ in statements that hold for all $\AGr\in\AG$.
\begin{theorem}
\label{thm:finite-complexity}
The complexity of every $\fTG$-polynomial system $\tilde\TGTvec = (\TGT_{0},\dots,\TGT_{j})$ is bounded by a constant $\cplx(\vd,j)$ that only depends on the length $\vd$ of the prefiltration $\fTG$ and the size $j$ of the system.
\end{theorem}
The proof is by induction on $d$.
For induction purposes we need the formally stronger statement below.
We use the convenient shorthand notation $\TGT(\TGS_{0},\dots,\TGS_{k})=(\TGT \TGS_{0},\dots,\TGT \TGS_{k})$.

\begin{proposition}
\label{prop:finite-complexity}
Let $\tilde\TGTvec = (\TGT_{0},\dots,\TGT_{j})$ be a $\fTG$-polynomial system.
Let also $\TGSvec_{0},\dots,\TGSvec_{j}$ be $\fTG[+1]$-polynomial systems and assume $\complexity\TGSvec_{j}\leq\cplx_{j}$.
Then the complexity of the system $\TGTvec = \TGT_{0}\TGSvec_{0} \uplus\dots\uplus \TGT_{j}\TGSvec_{j}$ is bounded by a constant $\cplx'=\cplx'(\vd,j,|\TGSvec_{0}|,\dots,|\TGSvec_{j-1}|,\cplx_{j})$, where $\vd$ is the length of $\fTG$.
\end{proposition}
The induction scheme is as follows.
Theorem~\ref{thm:finite-complexity} with length $\vd-1$ is used to prove Proposition~\ref{prop:finite-complexity} with length $\vd$, that in turn immediately implies Theorem~\ref{thm:finite-complexity} with length $\vd$.
The base case, namely Theorem~\ref{thm:finite-complexity} with $\vd=-\infty$, is trivial and $\cplx(-\infty,j)=0$.
\begin{proof}[Proof of Prop.~\ref{prop:finite-complexity} assuming Thm.~\ref{thm:finite-complexity} for length $\vd-1$]
It suffices to obtain a uniform bound on the complexity of $\TGTvec^{*}$ for every reduction $\TGTvec^{*}=\TGTvec^{*}_{\AGr}$, possibly cheating first.
Splitting $\TGSvec_{j} = \TGSvec_{j}' \uplus (\TGS)$ (where $\TGSvec_{j}'$ might be empty) we obtain
\begin{equation}
\label{eq:g-star}
\TGTvec^{*}
=
\TGT_{0}\TGSvec_{0} \uplus\dots\uplus \TGT_{j-1}\TGSvec_{j-1} \uplus \TGT_{j}\TGSvec_{j}'
\uplus
\reduction{\TGT_{j} \TGS}{\TGT_{0}\TGSvec_{0} \uplus\dots\uplus \TGT_{j-1}\TGSvec_{j-1} \uplus \TGT_{j}\TGSvec_{j}'}.
\end{equation}
Note that for every $\fTG[+1]$-polynomial $\TGS'$ we have
\begin{equation}
\label{eq:red-gjgj}
\reduction{\TGT_{j}\TGS}{\TGT_{j}\TGS'}
=
\TGT_{j}\TGS (\Trans \TGT_{j} \Trans \TGS)\inv \Trans \TGT_{j} \Trans \TGS'
=
\TGT_{j} \TGS (\Trans \TGS)\inv \Trans \TGS'
=
\TGT_{j} \reduction{\TGS}{\TGS'}
\end{equation}
and
\begin{multline}
\label{eq:red-gjgi}
\reduction{ \TGT_{j} \TGS }{ \TGT_{i}\TGS' }
= \Der(\TGS\inv \TGT_{j}\inv) \Trans \TGT_{i} \Trans \TGS'
= \Der(\TGS\inv \TGT_{j}\inv) \TGT_{i} \Der\TGT_{i} \Trans \TGS'\\
= \TGT_{i} \Der(\TGS\inv \TGT_{j}\inv) [\Der(\TGS\inv \TGT_{j}\inv), \TGT_{i}] \Der\TGT_{i} \Trans \TGS'
= \TGT_{i} \tilde \TGS,
\end{multline}
where $\tilde \TGS$ is a $\fTG[+1]$-polynomial by Theorem~\ref{thm:poly-group}.
By cheating we can rearrange the terms on the right-hand side of \eqref{eq:g-star}, obtaining
\begin{equation}
\label{eq:g-star-2}
\complexity\TGTvec^{*}
\leq
\complexity\left(
\TGT_{0}\tilde\TGSvec_{0} \uplus\dots\uplus \TGT_{j-1}\tilde\TGSvec_{j-1}
\uplus
\TGT_{j} \left( \TGSvec_{j}' \uplus \reduction{\TGS}{\TGSvec_{j}'} \right)
\right)
\end{equation}
for some $\fTG[+1]$-polynomial systems $\tilde\TGSvec_{0},\dots,\tilde\TGSvec_{j-1}$ with cardinality $2|\TGSvec_{0}|,\dots,2|\TGSvec_{j-1}|$, respectively.

We use nested induction on $j$ and $\cplx_{j}$.
In the base case $j=0$ we have $\TGTvec=\TGSvec_{0}$ and we obtain the conclusion with
\[
\cplx'(\vd,0,\cplx_{0})=\cplx_{0}.
\]
Suppose that $j>0$ and the conclusion holds for $j-1$.
If $\cplx_{j}=0$, then by cheating we may assume $\TGSvec_{j}=(1_{\TG})$.
Moreover, \eqref{eq:g-star-2} becomes
\[
\complexity\TGTvec^{*}
\leq
\complexity\left(
\TGT_{0}\tilde\TGSvec_{0} \uplus\dots\uplus \TGT_{j-1}\tilde\TGSvec_{j-1}
\right).
\]
The induction hypothesis on $j$ and Theorem~\ref{thm:finite-complexity} applied to $\tilde\TGSvec_{j-1}$ yield the conclusion with the bound
\begin{multline*}
\cplx'(\vd,j,|\TGSvec_{0}|,\dots,|\TGSvec_{j-1}|,0)
=
\cplx'(\vd,j-1,2|\TGSvec_{0}|,\dots,2|\TGSvec_{j-2}|,\cplx(\vd-1,2|\TGSvec_{j-1}|))+1.
\end{multline*}
If $\cplx_{j}>0$, then by cheating we may assume $\TGSvec_{j}\neq (1_{\TG})$ and $\complexity\TGSvec_{j}^{*} \leq \cplx_{j} - 1$, and \eqref{eq:g-star-2} becomes
\[
\complexity\TGTvec^{*}
\leq
\complexity\left(
\TGT_{0}\tilde\TGSvec_{0} \uplus\dots\uplus \TGT_{j-1}\tilde\TGSvec_{j-1}
\uplus \TGT_{j}\TGSvec_{j}^{*}
\right).
\]
The induction hypothesis on $\cplx_{j}$ now yields the conclusion with the bound
\[
\cplx'(\vd,j,|\TGSvec_{0}|,\dots,|\TGSvec_{j-1}|,\cplx_{j})
=
\cplx'(\vd,j,2|\TGSvec_{0}|,\dots,2|\TGSvec_{j-1}|,\cplx_{j}-1)+1.
\qedhere
\]
\end{proof}
\begin{proof}[Proof of Thm.~\ref{thm:finite-complexity} assuming Prop.~\ref{prop:finite-complexity} for length $\vd$]
Use Proposition~\ref{prop:finite-complexity} with system $\tilde\TGTvec$ as in the hypothesis and systems $\TGSvec_{0},\dots,\TGSvec_{j}$ being the trivial system $(1_{\TG})$.
This yields the bound
\[
\cplx(\vd,j) = \cplx'(\vd,j,1,\dots,1,0).
\qedhere
\]
\end{proof}

\section{The structure theorem}
\label{sec:structure}
The idea to prove a structure theorem for elements of a Hilbert space via the Hahn-Banach theorem is due to Gowers \cite[Proposition 3.7]{MR2669681}.
The insight of Walsh \cite[Proposition 2.3]{MR2912715} was to allow the ``structured'' and the ``pseudorandom'' part in the decomposition to take values in varying spaces that satisfy a monotonicity condition.

The assumption that these spaces are described by norms that are equivalent to the original Hilbert space norm can be removed.
In fact the structure theorem continues to hold for spaces described by extended seminorms\footnote{An \emph{extended seminorm} $\|\cdot\|$ on a vector space $H$ is a function with \emph{extended} real values $[0,+\infty]$ that is subadditive, homogeneous (i.e.\ $\|\lambda u\| = |\lambda| \|u\|$ if $\lambda\neq0$) and takes the value $0$ at $0$.}
\index{extended seminorm}
that are easier to construct in practice as we will see in Lemma~\ref{lem:Sigma-ext-seminorm}.

The Hahn-Banach theorem is used in the following form.
\begin{lemma}
\label{lem:sep}
Let $V_{i}$, $i=1,\dots,k$, be convex subsets of a Hilbert space $H$, at least one of which is open, and each of which contains $0$.
Let $V:=c_{1}V_{1}+\dots+c_{k}V_{k}$ with $c_{i}>0$ and take $f\not\in V$.
Then there exists a vector $\phi\in H$ such that $\<f,\phi\> \geq 1$ and $\<v,\phi\> < c_{i}\inv$ for every $v\in V_{i}$ and every $i$.
\end{lemma}
\begin{proof}
By the assumption the set $V$ is open, convex and does not contain $f$.
By the Hahn-Banach theorem there exists a $\phi\in H$ such that $\<f,\phi\> \geq 1$ and $\<v,\phi\> < 1$ for every $v\in V$.
The claim follows.
\end{proof}
The next result somewhat resembles Tao's structure theorem \cite{MR2274314}, though Tao's result gives additional information (positivity and boundedness of the structured part).
Information of that kind can also be extracted from the proof via the Hahn--Banach theorem using some of the more advanced techniques of Gowers \cite{MR2669681}.
\begin{structuretheorem}
\label{thm:structure}
For every $\delta>0$,
any functions $\omega,\psi\from\iset\to\iset$,
and every $M_{\bullet}\in\iset$ there exists an increasing sequence of indices
\begin{equation}
\label{eq:decomposition-seq}
M_{\bullet} \leq M_1 \leq \dots \leq M_{\lceil 2\delta^{-2} \rceil}
\end{equation}
for which the following holds.
Let $\eta \from\R_+\to\R_+$ be any function and \mbox{$(\|\cdot\|_\indx)_{\indx\in\iset}$} be a net of extended seminorms on a Hilbert space $H$ such that the net of dual extended seminorms $(\|\cdot\|_\indx^*)_{\indx\in\iset}$ decreases monotonically.
Then for every $f\in H$ with $\|f\| \leq 1$ there exists a decomposition
\begin{equation}
\label{eq:decomposition}
f = \sigma + u + v
\end{equation}
and an $1\leq i\leq \lceil 2\delta^{-2} \rceil$ such that
\begin{equation}
\|\sigma\|_\indy < C^{\delta,\eta}_i, \quad
\|u\|_\indx^* < \eta(C^{\delta,\eta}_i), \quad\text{and}\quad
\|v\| < \delta,
\end{equation}
where the indices $\indx$ and $\indy$ satisfy $\omega(\indx) \leq M_i$ and $\psi(M_i) \leq \indy$,
and where the constant $C^{\delta,\eta}_{i}$ belongs to a decreasing sequence that only depends on $\delta$ and $\eta$ and is defined inductively starting with
\begin{equation}
\label{eq:C}
C^{\delta,\eta}_{\lceil 2\delta^{-2} \rceil} = 1
\quad\text{by}\quad
C^{\delta,\eta}_{i-1} = \max \Big\{ C^{\delta,\eta}_{i}, \frac{2}{\eta(C^{\delta,\eta}_{i})} \Big\}.
\end{equation}
\end{structuretheorem}
In the sequel we will only use Theorem~\ref{thm:structure} with the identity function $\omega(\indx):=\indx$, in which case we can choose $\indx=M_i$, and with $\delta$ and $\eta$ as in \eqref{eq:eta}.
\begin{proof}
It suffices to consider functions such that $\omega(\indx) \geq \indx$ and $\psi(\indx) \geq \indx$ for all $\indx$ (in typical applications $\psi$ grows rapidly).

The sequence $(M_{i})$ and auxiliary sequences $(\indx_{i})$, $(\indy_{i})$ are defined inductively starting with $\indx_{1}:=M_{\bullet}$ by
\[
M_{i} := \omega(\indx_{i}),
\quad \indy_{i}:=\psi(M_{i}),
\quad \indx_{i+1}:=\indy_{i},
\]
so that all three sequences increase monotonically.
Let $r$ be chosen later and assume that there is no $i \in \{1,\dots,r\}$ for which a decomposition of the form \eqref{eq:decomposition} with $\indx=\indx_{i}$, $\indy=\indy_{i}$ exists.

For every $i \in \{1,\dots,r\}$ we apply Lemma~\ref{lem:sep} with $V_{1},V_{2},V_{3}$ being the open unit balls of $\|\cdot\|_{\indy_{i}}$, $\|\cdot\|_{\indx_{i}}^{*}$ and $\|\cdot\|$, respectively, and with $c_{1}=C_{i}$, $c_{2}=\eta(C_{i})$, $c_{3}=\delta$.
Note that $V_{3}$ is open in $H$.
We obtain vectors $\phi_{i} \in H$ such that
\[
\<\phi_{i},f\> \geq 1, \quad \|\phi_{i}\|_{\indy_{i}}^{*} \leq (C_{i})\inv, \quad \|\phi_{i}\|_{\indx_{i}}^{**} \leq \eta(C_{i})\inv, \quad \|\phi_{i}\| \leq \delta\inv.
\]
Take $i<j$, then $\indy_{i}\leq \indx_{j}$, and by \eqref{eq:C} we have
\begin{multline*}
|\<\phi_{i},\phi_{j}\>|
\leq \|\phi_{i}\|_{\indx_{j}}^{*} \|\phi_{j}\|_{\indx_{j}}^{**}
\leq \|\phi_{i}\|_{\indy_{i}}^{*} \|\phi_{j}\|_{\indx_{j}}^{**}\\
\leq (C_{i})\inv \eta(C_{j})\inv
\leq (C_{j-1})\inv \eta(C_{j})\inv
\leq (2 \eta(C_{j})\inv)\inv \eta(C_{j})\inv
= \frac12,
\end{multline*}
so that
\[
r^{2} \leq \<\phi_{1}+\dots+\phi_{r},f\>^{2}
\leq \|\phi_{1}+\dots+\phi_{r}\|^{2}
\leq r \delta^{-2} + \frac{r^{2}-r}{2},
\]
which is a contradiction if $r \geq 2 \delta^{-2}$.
\end{proof}

\section{Reducible and structured functions}
\label{sec:reducible}
In this section we adapt Walsh's notion of a structured function and a corresponding inverse theorem to the context of amenable groups.
Informally, a function is reducible with respect to a system if its shifts can be approximated by shifts arising from reductions of this system, uniformly over \folner{} sets that are not too large.
A function is structured if it is a linear combination of reducible functions.

\begin{definition}
\label{def:uniformly-N-reducible}
Let $\TGTvec=(\TGT_{0},\dots,\TGT_{j})$ be a system, $\gamma>0$ and $\indx \in \iset$.
A function $\sigma$ bounded by one is called \emph{uniformly $(\TGTvec, \gamma, \indx)$-reducible}
\index{uniformly reducible function}
(in symbols $\sigma\in\Sigma_{\TGTvec, \gamma, \indx}$)
if for every \folner{} set $I$ with $\varphi_{\gamma}(\lfloor I \rfloor) \leq \indx$ there exist
functions $b_0, \dots, b_{j-1}$ bounded by one and a finite measure set $J \subset \AG$ such that for every $\AGg \in I$
\begin{equation}
\label{ineq:unif-red}
\Big\| \TGT_j(\AGg) \sigma - \E_{\AGh \in J} \prod_{i=0}^{j-1} \reduction{\TGT_j}{\TGT_i}{\AGh}(\AGg) b_i \Big\|_\infty
< \gamma.
\end{equation}
\end{definition}
Walsh's definition of $L$-reducibility with parameter $\epsilon$ corresponds to uniform $(\TGTvec, \gamma, \indx)$-reducibility with $\indx=\varphi_{\gamma}(L)$ and a certain $\gamma=\gamma(\epsilon)$ that will now be defined along with other parameters used in the proof of the main result.

Given $\epsilon>0$ we fix
\begin{equation}
\label{eq:eta}
\delta = \frac{\epsilon}{2^{2}\cdot 3^{2}}
\quad \text{and}
\quad \eta(x) = \frac{\epsilon^{2}}{2^{3}\cdot 3^{3} x}
% so that the Inverse Theorem~\ref{thm:inverse} works
\end{equation}
and define the decreasing sequence
$C_{1}^{\delta,\eta} \geq \dots \geq C_{\lceil 2\delta^{-2} \rceil}^{\delta,\eta}$
as in \eqref{eq:C}.
It is in turn used to define the function
\begin{equation}
\label{eq:gamma}
\gamma = \gamma^{1}(\epsilon) = \frac{\epsilon}{3\cdot 8 C^{*}},
\quad\text{where}\quad C^{*} = C_{1}^{\delta,\eta},
\end{equation}
and its iterates $\gamma^{\cplx+1}(\epsilon) = \gamma^{\cplx}(\gamma)$.

The ergodic average corresponding to a system $\TGTvec = (\TGT_{0},\dots,\TGT_{j})$, bounded functions $f_{0},\dots, f_{j}$ and a finite measure set $I \subset \AG$ is denoted by
\[
\Av{I}[f_0, \dots, f_j]
:=
\E_{\AGg\in I} \prod_{i=0}^j \TGT_i(\AGg) f_i.
\]
The inverse theorem below tells that any function that gives rise to a large ergodic average correlates with a reducible function.
\begin{inversetheorem}
\label{thm:inverse}
Let $\epsilon > 0$.
Suppose that $\|u\|_\infty \leq 3C$, the functions $f_0, \dots, f_{j-1}$ are bounded by one, and $\| \Av{I}[f_0, \dots, f_{j-1}, u] \|_2 > \epsilon/6$ for some \folner{} set $I = \Fo_\indx \AGr$.
Then there exists a uniformly $(\TGTvec, \gamma, \indx)$-reducible function $\sigma$ such that $\<u,\sigma\> > 2\eta(C)$.
\end{inversetheorem}
\begin{proof}
Set $b_0 := \Av{I}[f_0, \dots, f_{j-1}, u] f_{0} / \|u\|_\infty$, so that $\|b_0\|_\infty \leq 1$, and $b_{1}:=f_{1},\dots,b_{j-1}:=f_{j-1}$.
Recall $\TGT_{0}\equiv 1_{\TG}$ and note that
\begin{align*}
2\eta(C)
&<
\|u\|_\infty\inv \left\| \Av{I} [f_0, \dots, f_{j-1}, u] \right\|_2^2\\
&=
\< \E_{\AGg\in I} \prod_{i=0}^{j-1} \TGT_i(\AGg) f_i \cdot \TGT_j(\AGg) u, \frac{\Av{I}[f_0, \dots, f_{j-1}, u]}{\|u\|_\infty} \>\\
&=
\E_{\AGg\in I} \< \TGT_j(\AGg) u, \prod_{i=0}^{j-1} \TGT_i(\AGg) b_i \>\\
&=
\< u, \underbrace{\E_{\AGh\in \Fo_{\indx}} \prod_{i=0}^{j-1} \TGT_j(\AGh\AGr)\inv \TGT_i(\AGh\AGr) b_i}_{=: \sigma} \>.
\end{align*}
We claim that $\sigma$ is uniformly $(\TGTvec, \gamma, \indx)$-reducible.

Consider a \folner{} set $I = \Fo_L \tilde \AGr$ with $\varphi_{\gamma}(L) \leq \indx$.
We have to show \eqref{ineq:unif-red} for some $J\subset\AG$ and every element of $I$.
Let $\AGg\in \Fo_L$.
By definition \eqref{eq:varphi} of $\varphi_{\gamma}$ we obtain
\[
\Big\| \sigma - \E_{\AGh\in \Fo_\indx} \prod_{i=0}^{j-1} \TGT_j(\AGg\AGh\AGr)\inv \TGT_i(\AGg\AGh\AGr) b_i \Big\|_\infty
\leq \frac{\meas{\AGg \Fo_\indx \Delta \Fo_\indx}}{\meas{\Fo_\indx}} < \gamma.
\]
Since $\TGT_j(\AGg\tilde \AGr)$ is an isometric algebra homomorphism, we see that $\TGT_j(\AGg\tilde \AGr) \sigma$ is uniformly approximated by
\[
\E_{\AGh\in \Fo_\indx} \prod_{i=0}^{j-1} \TGT_j(\AGg\tilde \AGr) \TGT_j(\AGg\AGh\AGr)\inv \TGT_i(\AGg\AGh\AGr) b_i.
\]
Splitting $\AGg\AGh\AGr = \AGg\tilde \AGr \cdot \tilde \AGr\inv \AGh\AGr$ we can write this function as
\[
\E_{\AGh\in \Fo_\indx} \prod_{i=0}^{j-1} \reduction{\TGT_j}{\TGT_i}{\tilde \AGr\inv \AGh\AGr}(\AGg\tilde \AGr) b_i
=
\E_{\AGh\in \tilde \AGr\inv \Fo_\indx \AGr} \prod_{i=0}^{j-1} \reduction{\TGT_j}{\TGT_i}{\AGh}(\AGg\tilde \AGr) b_i,
\]
which gives \eqref{ineq:unif-red} with $J=\tilde \AGr\inv \Fo_\indx \AGr$.
\end{proof}
Structure will be measured by extended seminorms associated to sets $\Sigma$ of reducible functions by the following easy lemma.
\begin{lemma}[cf. {\cite[Corollary 3.5]{MR2669681}}]
\label{lem:Sigma-ext-seminorm}
Let $H$ be an inner product space and $\Sigma\subset H$.
Then the formula
\begin{equation}
\|f\|_{\Sigma} := \inf\Big\{ \sum_{t=0}^{k-1}|\lambda_{t}| : f=\sum_{t=0}^{k-1}\lambda_{t}\sigma_{t}, \sigma_{t} \in \Sigma \Big\},
\end{equation}
where empty sums are allowed and the infimum of an empty set is by convention $+\infty$, defines an extended seminorm on $H$ whose dual extended seminorm is given by
\begin{equation}
\|f\|_{\Sigma}^{*} := \sup_{\phi\in H: \|\phi\|_{\Sigma}\leq 1} |\<f,\phi\>| = \sup_{\sigma\in\Sigma}|\<f,\sigma\>|.
\end{equation}
\end{lemma}
Heuristically, a function with small dual seminorm is pseudorandom since it does not correlate much with structured functions.

\section{Metastability of averages for finite complexity systems}
\label{sec:metastability}
We come to the proof of the norm convergence result.
Walsh's approach to this problem involves a quantitative statement that is stronger than Theorem~\ref{thm:norm-convergence} in the same way as the quantitative von Neumann Theorem~\ref{thm:vn-quan} is stronger than the finitary von Neumann Theorem~\ref{thm:vn-fin}.
We use the notation
\[
\Av{I,I'}[f_0, \dots, f_j]
:=
\Av{I}[f_0, \dots, f_j] - \Av{I'}[f_0, \dots, f_j].
\]
for the difference between two multiple averages.
\begin{theorem}
\label{thm:metastability}
For every complexity $\cplx\in\N$ and every $\epsilon>0$ there exists $K_{\cplx,\epsilon}\in\N$ such that
for every function $F \from \iset \to \iset$ and every $M \in \iset$ there exists a tuple of indices
\begin{equation}
\label{eq:main-thm-seq}
M\leq M^{\cplx,\epsilon,F}_1 , \dots , M^{\cplx,\epsilon,F}_{K_{\cplx,\epsilon}} \in \iset
\end{equation}
of size $K_{\cplx,\epsilon}$
such that for every system $\TGTvec$ with complexity at most $\cplx$ and every choice of functions $f_0, \dots, f_j \in L^{\infty}(X)$ bounded by one there exists $1 \leq i \leq K_{\cplx,\epsilon}$ such that
for all \folner{} sets $I,I'$ with $M^{\cplx,\epsilon,F}_i \leq \lfloor I\rfloor,\lfloor I'\rfloor$ and $\lceil I, I'\rceil_{\gamma^{\cplx}(\epsilon)} \leq F(M^{\cplx,\epsilon,F}_i)$ we have
\begin{equation}
\label{eq:main-thm-est}
\| \Av{I,I'}[f_0, \dots, f_j] \|_2 < \epsilon.
\end{equation}
\end{theorem}
Recall that $\lceil I,I' \rceil_{\gamma^{\cplx}(\epsilon)}$ was defined in Lemma~\ref{lem:ceil}.
Theorem~\ref{thm:metastability} will be proved by induction on the complexity $\cplx$.
As an intermediate step we need the following.
\begin{proposition}
\label{prop:metastability}
For every complexity $\cplx\in\N$ and every $\epsilon>0$ there exists $\tilde K_{\cplx,\epsilon}\in\N$ such that
for every function $F \from \iset \to \iset$ and every $\tilde M \in \iset$ there exists a tuple of indices
\begin{equation}
\label{eq:main-prop-seq}
\tilde M \leq \tilde M^{\cplx,\epsilon,F}_1 , \dots , \tilde M^{\cplx,\epsilon,F}_{\tilde K_{\cplx,\epsilon}} \in \iset
\end{equation}
of size $\tilde K_{\cplx,\epsilon}$
as well as an index $\indx = \indx_{\cplx,\epsilon,F}(\tilde M)$ such that the following holds.
For every system $\TGTvec$ such that every reduction $\TGTvec^*_{\AGr}$ ($\AGr\in\AG$) has complexity at most $\cplx$,
every choice of functions $f_0, \dots, f_{j-1} \in L^{\infty}(X)$ bounded by one,
and every finite linear combination $\sum_{t} \lambda_t \sigma_t$ of uniformly $(\TGTvec, \gamma, \indx)$-reducible functions $\sigma_{t}$
there exists $1 \leq \tilde i \leq \tilde K_{\cplx,\epsilon}$ such that
for all \folner{} sets $I,I'$ with $\tilde M^{\cplx,\epsilon,F}_{\tilde i} \leq \lfloor I\rfloor,\lfloor I'\rfloor$ and $\lceil I, I'\rceil_{\gamma^{\cplx+1}(\epsilon)} \leq F(\tilde M^{\cplx,\epsilon,F}_{\tilde i})$ we have
\begin{equation}
\Big\| \Av{I,I'}[f_0, \dots, f_{j-1},\sum_{T} \lambda_t \sigma_t] \Big\|_2 < 8 \gamma \sum_t |\lambda_t|.
\end{equation}
\end{proposition}
The induction procedure is as follows.
Theorem~\ref{thm:metastability} for complexity $\cplx$ is used to deduce Proposition~\ref{prop:metastability} for complexity $\cplx$, which is in turn used to show Theorem~\ref{thm:metastability} for complexity $\cplx+1$.
The base case (Theorem~\ref{thm:metastability} with $\cplx=0$) is trivial, take $K_{0,\epsilon}=1$ and $M_{1}^{0,\epsilon,F}=M$.

In order to keep an overview we note that in the proofs below elements of $\AG$ are denoted by $\AGl,\AGr,\AGg,\AGh$, indices (elements of $\iset$) by $\indx,\indy,M$, real numbers by $\epsilon,\delta,C$, integers by $i,j,k,t,K,\cplx$, and real-valued functions on $X$ by $f,\sigma,u,v$.

\begin{proof}[Proof of Prop.~\ref{prop:metastability} assuming Thm.~\ref{thm:metastability} for complexity $\cplx$]
The tuple \eqref{eq:main-prop-seq} and the index $\indx$ will be chosen later.
For the moment assume that $I, I' \lesssim_{\gamma} I_{0}$ for some \folner{} set $I_{0}$ with $\varphi_{\gamma}(\lfloor I_{0}\rfloor) \leq \indx$.
Consider the functions $b_{0}^{t},\dots,b_{j-1}^{t}$ bounded by one and the sets $J^{t} \subset \AG$ from the definition of uniform $(\TGTvec, \gamma, \indx)$-reducibility of $\sigma_{t}$ over $I_0$ (Definition~\ref{def:uniformly-N-reducible}).
Write $O(x)$ for an error term bounded by $x$ in $L^\infty(X)$.
By \eqref{ineq:unif-red} we have
\begin{multline*}
\Av{I}[f_0, \dots, f_{j-1}, \sigma_t]
= \meas{I}\inv \int_{\AGg\in I} \prod_{i=0}^{j-1} \TGT_i(\AGg) f_i \cdot \TGT_j(\AGg)\sigma_t\\
= \meas{I}\inv \int_{\AGg\in I\cap I_0} \prod_{i=0}^{j-1} \TGT_i(\AGg) f_i
\left( \E_{\AGr\in J^t} \prod_{i=0}^{j-1} \reduction{\TGT_j}{\TGT_i}{\AGr}(\AGg) b_i^t + O(\gamma) \right)
+ \frac{\meas{I\setminus I_0}}{\meas{I}} O(1).
\end{multline*}
The first error term accounts for the $L^{\infty}$ error in the definition of uniform reducibility
and the second for the fraction of $I$ that is not contained in $I_{0}$.
This can in turn be approximated by
\begin{multline*}
\meas{I}\inv \int_{\AGg\in I} \E_{\AGr\in J^t} \prod_{i=0}^{j-1} \TGT_i(\AGg) f_i
\prod_{i=0}^{j-1} \reduction{\TGT_j}{\TGT_i}{\AGr}(\AGg) b_i^t
+ \frac{\meas{I\cap I_0}}{\meas{I}} O(\gamma)
+ \frac{\meas{I\setminus I_0}}{\meas{I}} O(2)\\
= \E_{\AGr\in J^{t}} \Av[\TGTvec^{*}_{\AGr}]{I}[f_0, \dots, f_{j-1}, b_0^t, \dots, b_{j-1}^t] + O(3\gamma).
\end{multline*}
Using the analogous approximation for $I'$ and summing over $t$ we obtain
\begin{multline}
\label{eq:est-prop}
\| \Av{I,I'}[f_0, \dots, f_{j-1},\sum_{t} \lambda_t \sigma_t] \|_{2}\\
\leq
\sum_{t} |\lambda_{t}| \E_{\AGr\in J^{t}} \| \Av[\TGTvec^{*}_{\AGr}]{I,I'}[f_0, \dots, f_{j-1},b_{0}^{t},\dots,b_{j-1}^{t}] \|_{2}
+6 \gamma \sum_{t} |\lambda_{t}|.
\end{multline}
If $\TG$ is commutative and $\TGTvec$ consists of affine mappings, then the maps that constitute systems $\TGTvec^{*}_{\AGr}$ differ at most by constants, and in this case one can bound the first summand by a norm of a difference of averages associated to certain functions on $X \times \uplus_{t} J^{t}$ similarly to the reduction in \cite[\textsection 5]{MR2408398}.
In general we need (a version of) the more sophisticated argument of Walsh that crucially utilizes the uniformity in the induction hypothesis.
The argument provides a bound on most (with respect to the weights $|\lambda_{t}|/\meas{J^{t}}$) of the norms that occur in the first summand.

Let $r = r(\cplx,\epsilon)$ be chosen later.
We use the operation $M \mapsto M^{\cplx,\gamma,F_{s}}_i$ and the constant $K=K_{\cplx,\gamma}$ from Theorem~\ref{thm:metastability} (with $\gamma$ in place of $\epsilon$) to inductively define functions $F_{r},\dots,F_{1} \from\iset\to\iset$ by
\[
F_{r}=F, \quad F_{s-1}(M) := \sup_{1 \leq i \leq K} F_{s}(M^{\cplx,\gamma,F_{s}}_i).
\]
This depends on a choice of a supremum function for the directed set $\iset$ that can be made independently of all constructions performed here.
Using the same notation define inductively for $1\leq i_1,\dots,i_r \leq K$ the indices
\[
\tilde M^{()} := \tilde M, \quad
\tilde M^{(i_{1},\dots,i_{s-1},i_{s})} :=
(\tilde M^{(i_{1},\dots,i_{s-1})})^{\cplx,\gamma,F_{s}}_{i_{s}}.
\]
The theorem tells that for every $t$, $\AGr$ and $1 \leq i_{1},\dots,i_{s-1} \leq K$
there exists some $1 \leq i_{s} \leq K$ such that
\begin{equation}
\label{eq:est-red}
\| \Av[\TGTvec^{*}_{\AGr}]{I,I'}[f_0, \dots, f_{j-1},b_{0}^{t},\dots,b_{j-1}^{t}] \|_{2} <\gamma
\end{equation}
holds provided
\begin{equation}
\label{eq:main-ind-cond}
\tilde M^{(i_{1},\dots,i_{s})} = (\tilde M^{(i_{1},\dots,i_{s-1})})^{\cplx,\gamma,F_{s}}_{i_{s}} \leq \lfloor I\rfloor, \lfloor I'\rfloor
\text{ and }
\lceil I, I'\rceil_{\gamma^{\cplx}(\gamma)} \leq F_{s}(\tilde M^{(i_{1},\dots,i_{s})}).
\end{equation}

Start with $s=1$.
By the pigeonhole principle there exists an $i_1$ such that \eqref{eq:est-red} holds for at least the fraction $1/K$ of the pairs $(t,\AGr)$ with respect to the weights $|\lambda_{t}|/\meas{J^{t}}$ (provided \eqref{eq:main-ind-cond} with $s=1$).

Using the pigeonhole principle repeatedly on the remaining pairs $(t,\AGr)$ with weights $|\lambda_{t}|/\meas{J^{t}}$ we can find a sequence $i_{1},\dots, i_{r}$ such that for all pairs but the fraction $(\frac{K-1}{K})^{r}$ the estimate \eqref{eq:est-red} holds provided that the conditions \eqref{eq:main-ind-cond} are satisfied for all $s$.

By definition we have
$\tilde M \leq \tilde M^{(i_{1})} \leq \tilde M^{(i_{1},i_{2})} \leq \dots \leq \tilde M^{(i_{1},\dots,i_{r})}$
and
\begin{multline*}
F_{1}(\tilde M^{(i_{1})})
\geq F_{2}((\tilde M^{(i_{1})})^{\cplx,\gamma,F_{2}}_{i_{2}})
= F_{2}(\tilde M^{(i_{1},i_{2})})
\geq \dots\\
\geq F_{r}(\tilde M^{(i_{1},\dots,i_{r})})
= F(\tilde M^{(i_{1},\dots,i_{r})})
\end{multline*}
for any choice of $i_{1},\dots,i_{r}$.
Therefore the conditions \eqref{eq:main-ind-cond} become stronger as $s$ increases.
Recall from \eqref{eq:gamma} that $\gamma^{\cplx}(\gamma) = \gamma^{\cplx+1}(\epsilon)$, thus we only need to ensure
\begin{equation}
\label{eq:main-ind-cond2}
\tilde M^{(i_{1},\dots,i_{r})} \leq \lfloor I\rfloor,\lfloor I'\rfloor
\text{ and }
\lceil I, I'\rceil_{\gamma^{\cplx+1}(\epsilon)} \leq F(\tilde M^{(i_{1},\dots,i_{r})}).
\end{equation}
This is given by the hypothesis if we define the tuple \eqref{eq:main-prop-seq} to consist of all numbers $\tilde M^{(i_{1},\dots,i_{r})}$ where $i_{1},\dots,i_{r} \in \{1,\dots,K\}$, so $\tilde K_{\cplx,\epsilon} = (K_{\cplx,\gamma})^{r}$.

We now choose $r$ to be large enough that $(\frac{K-1}{K})^{r} < \gamma$.
% It suffices to have $r > -K\ln\gamma = -\ln\gamma / (1/K) > -\ln\gamma / (\ln K - \ln(K-1))$
Then the sum at the right-hand side of \eqref{eq:est-prop} splits into a main term that can be estimated by $\gamma\sum_t |\lambda_t|$ using \eqref{eq:est-red} and an error term that can also be estimated by $\gamma\sum_t |\lambda_t|$ using the trivial bound
\[
\| \Av[\TGTvec^{*}_{\AGr}]{I,I'}[f_0, \dots, f_{j-1},b_{0}^{t},\dots,b_{j-1}^{t}] \|_{2} \leq 1.
\]
Finally, the second condition in \eqref{eq:main-ind-cond2} by definition means that there exists a \folner{} set $I_0$ such that $\lfloor I_0 \rfloor = F(\tilde M^{(i_{1},\dots,i_{r})})$ and $I,I' \lesssim_{\gamma^{\cplx+1}(\epsilon)} I_0$. In particular we have $I,I' \lesssim_{\gamma} I_0$ since $\gamma^{\cplx+1}(\epsilon) \leq \gamma^{1}(\epsilon) = \gamma$.
Taking
\[
\indx := \sup_{1\leq i_{1},\dots,i_{r} \leq K} \varphi_{\gamma}(F(\tilde M^{(i_{1},\dots,i_{r})}))
\]
guarantees $\varphi_{\gamma}(\lfloor I_0 \rfloor) \leq \indx$.
\end{proof}

\begin{proof}[Proof of Thm.~\ref{thm:metastability} assuming Prop.~\ref{prop:metastability} for complexity $\cplx-1$]
Let $\cplx$, $\epsilon$, $F$ and a system $\TGTvec$ with complexity at most $\cplx$ be given.
By cheating we may assume that every reduction $\TGTvec^*_{\AGr}$ ($\AGr\in\AG$) has complexity at most $\cplx-1$.

We apply the Structure Theorem~\ref{thm:structure} with the following data.
The extended seminorms $\|\cdot\|_{\indx} := \|\cdot\|_{\Sigma_{\TGTvec, \gamma, \indx}}$, $\indx\in\iset$, are given by Lemma~\ref{lem:Sigma-ext-seminorm}; the dual extended seminorms $\|\cdot\|_{\indx}^{*} = \|\cdot\|_{\Sigma_{\TGTvec, \gamma, \indx}}^{*}$ decrease monotonically since $\Sigma_{\TGTvec, \gamma, \indx'} \subset \Sigma_{\TGTvec, \gamma, \indx}$ whenever $\indx' \geq \indx$.
The function $\psi(\tilde M) := \indx_{\cplx,\epsilon,F}(\tilde M)$ is given by Proposition~\ref{prop:metastability} with $\cplx$, $\epsilon$, $F$ as in the hypothesis of this theorem.
Finally, $\omega(\indx):=\indx$ and $M_{\bullet}:=M$.
The structure theorem provides a decomposition
\begin{equation}
\label{eq:dec-fj}
f_{j} = \sum_{t}\lambda_{t}\sigma_{t} + u + v,
\end{equation}
where $\sum_{t}|\lambda_{t}| < C_{i}^{\delta,\eta} =: C_{i} \leq C^{*}$, $\sigma_{t} \in \Sigma_{\TGTvec, \gamma, B}$, $\|u\|_{M_i}^{*} < \eta(C_{i})$ and $\|v\|_{2} < \delta$.
Here $\psi(M_{i}) \leq B$, and the index $M_{i} \geq M_{\bullet}=M$ comes from the sequence \eqref{eq:decomposition-seq} that depends only on $\psi$, $M$ and $\epsilon$, and whose length $\lceil 2\delta^{-2} \rceil$ depends only on $\epsilon$.
Note that $\psi$ in turn depends only on $\cplx$, $\epsilon$ and $F$.

We will need an $L^{\infty}$ bound on $u$ in order to use the Inverse Theorem~\ref{thm:inverse}.
To this end let $S = \{ |v| \leq C_{i}\} \subset X$, then
\[
|u| 1_{S} \leq 1_{S} + \sum_{t}|\lambda_{t}|1_{S} + |v| 1_{S} \leq 3 C_{i}.
\]
Moreover, the restriction of $u$ to $S^{\complement}$ is bounded by
\[
|u| 1_{S^{\complement}} \leq 1_{S^{\complement}} + \sum_{t}|\lambda_{t}| 1_{S^{\complement}} +  |v| 1_{S^{\complement}} \leq 3 |v| 1_{S^{\complement}},
\]
so it can be absorbed in the error term $v$.
It remains to check that $\|u1_{S}\|_{M_i}^{*}$ is small.
By Chebyshev's inequality we have $C_{i}^{2} \mu(S^{\complement}) \leq \|v\|_{2}^{2}$, so that $\mu(S^{\complement})^{1/2} \leq \delta/C_{i}$.
Let now $\sigma \in \Sigma_{M_i}$ be arbitrary and estimate
\begin{multline*}
|\<u1_{S},\sigma\>|
\leq
|\<u,\sigma\>| + |\<u1_{S^{\complement}},\sigma\>|
\leq
\|u\|_{M_i}^{*} + \|u1_{S^{\complement}}\|_{2}\|\sigma 1_{S^{\complement}}\|_{2}\\
<
\eta(C_{i}) + 3 \|v\|_{2} \mu(S^{\complement})^{1/2}
\leq
\eta(C_{i}) + 3 \delta \cdot \delta/C_{i}
< 2\eta(C_{i}).
\end{multline*}
Thus (replacing $u$ by $u 1_{S}$ and $v$ by $v+u1_{S^{\complement}}$ if necessary) we may assume $\|u\|_{\infty} \leq 3 C_{i}$
at the cost of having only $\|u\|_{M_i}^{*} < 2\eta(C_{i})$ and $\|v\|_{2} \leq 4\delta < \epsilon/6$.

Now we estimate the contributions of the individual summands in \eqref{eq:dec-fj} to \eqref{eq:main-thm-est}.
The bounds
\[
\| \Av{I}[f_0, \dots, f_{j-1}, v] \|_2 \leq \frac\epsilon6
\quad\text{and}\quad
\| \Av{I'}[f_0, \dots, f_{j-1}, v] \|_2 \leq \frac\epsilon6
\]
are immediate.
Proposition~\ref{prop:metastability} for complexity $\cplx-1$ with $\tilde M = M_{i}$ (applicable since the functions $\sigma_{t}$ are uniformly $(\TGTvec, \gamma, \psi(M_{i}))$-reducible) shows that
\[
\Big\| \Av{I,I'}[f_0, \dots, f_{j-1}, \sum_{t}\lambda_{t}\sigma_{t}] \Big\|_2
< 8 \gamma \sum_t |\lambda_t| < \frac\epsilon3,
\]
provided that the \folner{} sets $I,I'$ satisfy
\[
\tilde M^{\cplx-1,\epsilon,F}_{\tilde i} \leq \lfloor I\rfloor,\lfloor I'\rfloor
\text{ and }
\lceil I, I'\rceil_{\gamma^{\cplx}(\epsilon)} \leq F(\tilde M^{\cplx-1,\epsilon,F}_{\tilde i})
\]
for some $\tilde M^{\cplx-1,\epsilon,F}_{\tilde i}$ that belongs to the tuple \eqref{eq:main-prop-seq} given by the same proposition.
The former condition implies in particular $M_{i} \leq \tilde M^{\cplx-1,\epsilon,F}_{\tilde i} \leq \lfloor I\rfloor, \lfloor I'\rfloor$,
and in this case the Inverse Theorem~\ref{thm:inverse} shows that
\[
\| \Av{I}[f_0, \dots, f_{j-1}, u] \|_2 \leq \frac\epsilon6
\quad\text{and}\quad
\| \Av{I'}[f_0, \dots, f_{j-1}, u] \|_2 \leq \frac\epsilon6,
\]
since otherwise there exists a uniformly $(\TGTvec, \gamma, M_i)$-reducible function $\sigma$ such that $\<u,\sigma\> > 2\eta(C_i)$.

We obtain the conclusion of the theorem with the tuple \eqref{eq:main-thm-seq} being the concatenation of the tuples \eqref{eq:main-prop-seq} provided by Proposition~\ref{prop:metastability} with $\tilde M=M_{i} \geq M$ for $1 \leq i \leq \lceil 2\delta^{-2} \rceil$.
In particular, $K_{\epsilon,\cplx} = \lceil 2\delta^{-2} \rceil \tilde K_{\epsilon,\cplx-1}$.
\end{proof}
This completes the induction and thus the proof of Proposition~\ref{prop:metastability} and Theorem~\ref{thm:metastability}.
The proof of the fact that metastability implies convergence has been already outlined in the discussion of the von Neumann mean ergodic theorem.
For completeness we repeat the full argument.
\begin{proof}[Proof of Theorem~\ref{thm:norm-convergence}]
In the case \eqref{thm:norm-convergence:polynomial} we apply Theorem~\ref{thm:finite-complexity} and in the case \eqref{thm:norm-convergence:commuting} Proposition~\ref{prop:comm-antihom}.
In both cases we obtain that the complexity $\cplx:=\complexity\TGTvec$ is finite.
We may assume that the functions $f_0,\dots,f_j$ are bounded by one.

Assume that the functions $\Av{I}[f_0, \dots, f_j]$ do not converge in $L^{2}(X)$ along $\lfloor I \rfloor\in\iset$.
Then there exists an $\epsilon>0$ such that for every $M\in\iset$ there exist \folner{} sets $I,I'$ such that $M \leq \lfloor I \rfloor, \lfloor I' \rfloor$ and
\[
\|\Av{I}[f_0, \dots, f_j] - \Av{I'}[f_0, \dots, f_j]\|_{2} = \| \Av{I,I'}[f_0, \dots, f_j] \|_2 > \epsilon.
\]
This contradicts Theorem~\ref{thm:metastability} with $F(M) := \lceil I, I' \rceil_{\gamma^{\cplx}(\epsilon)}$.
Therefore the limit
\[
\lim_{\lfloor I \rfloor \in\iset} \Av{I}[f_0, \dots, f_j]
\]
exists.
Since the \folner{} net was arbitrary, Lemma~\ref{lem:Clim-indep} implies that the limit does not depend on it.
\end{proof}
\printbibliography

\end{document}